\documentclass{amsart}

\usepackage{amssymb}
\usepackage{mathabx}
\usepackage{tikz}
\usepackage{rotating}
\usepackage{ytableau}
\usepackage{mathptmx}     
\usepackage{helvet}   
\usepackage{courier}
\usepackage{amsfonts} 
\usepackage{forest}

\usepackage{mathtools}

\newtheorem{theorem}{Theorem}[section]

\theoremstyle{definition}
\newtheorem{definition}[theorem]{Definition}

\theoremstyle{proposition}
\newtheorem{proposition}[theorem]{Proposition}

\theoremstyle{corollary}
\newtheorem{corollary}[theorem]{Corollary}

\numberwithin{equation}{section}

\begin{document}

\title[Lattice path enumeration]%
{Lattice path enumeration for semi-magic squares\\ by Latin rectangles} 


\author{Robert W. Donley, Jr. and Won Geun Kim}
\address{Department of Mathematics and Computer Science,  Queensborough Community College (CUNY), Bayside, NY 11364, USA}
\email{RDonley@qcc.cuny.edu}
\address{Department of Mathematics, Lander College for Men - Touro College and University System, Kew Garden Hills, NY 11367, USA}
\email{won-geun.kim2@touro.edu}



\begin{abstract}  
Similar to how standard Young tableaux represent paths in the Young lattice, Latin rectangles may be use to enumerate paths in the poset of semi-magic squares with entries zero or one.  The symmetries associated to determinant preserve this poset, and we completely describe the orbits, covering data, and  maximal chains for squares of size 4, 5, and 6.  The last item gives the number of Latin squares in these cases.  To calculate efficiently for size 6, we in turn identify orbits with certain equivalence classes of hypergraphs.
\end{abstract}

\maketitle

\section{Introduction}

In the Young lattice, paths starting at the minimum element $\hat{0}$ may be identified with standard Young tableaux in the following manner (for instance, \cite{StA}, Ch. 8):
$$
\hat{0}\ \ \to\ \
\begin{ytableau}
\ 1
\end{ytableau}
\  \ \to\   \
\begin{ytableau}
\ 1 & 2
\end{ytableau}
\ \ \to\ \
\begin{ytableau}
\ 1 & 2\\
\ 3 & \none
\end{ytableau}
\ \ \to\ \  
\begin{ytableau}
\ 1 & 2 & 4\\
\ 3 & \none
\end{ytableau}
\ \ \to \ \
\begin{ytableau}
\ 1 & 2 & 4\\
\ 3 &  5 & \none
\end{ytableau}.
$$
In turn, the standard tableaux for a given shape are enumerated by the hook length formula.  

A similar identification of paths may be made in the finite graded poset of semi-magic squares of size $n$.  By the Birkhoff-von Neumann theorem  (\cite{Bkf}, \cite{VN}), every semi-magic square may be constructed by iterated addition of permutation matrices;  if we represent these permutation matrices in single line notation, a path of semi-magic squares is given by a list of these lines, and, if the entries of the semi-magic square are zero or one, these lists form Latin rectangles.  

In \cite{Do}, basic properties of this poset were considered for size three.  For small order, it is reasonable to display the poset diagram of orbits under a familiar group action, and  relevant poset data may be summarized efficiently using homogeneity.  

Key to this work is the ability to draw connections between the following objects:
\begin{enumerate}
\item semi-magic squares with entries zero or one, 
\item Latin rectangles and Latin squares, and
\item hypergraphs that are both $k$-uniform and $k$-regular. 
\end{enumerate} 
For methodology, we favor the language of groups, permutation matrices, and semi-magic squares. We leave it to the reader  to recast the language for the other poset models when not explicit. For example, cycle switching for Latin rectangles (for instance, \cite{Wa1}) changes the representation of a semi-magic square as a sum of permutation matrices, and total path numbers in a given rank correspond to Latin rectangle counts, for which many formulas exist. Other related issues of interest include the face structure of the Birkhoff polytope (for instance, \cite{BS}, or Chapters 8 and 9 in \cite{Bru}) and the role of bipartite graphs there \cite{Phz}. 

In Section 2, we recall basic properties of semi-magic squares, the partial ordering, and the group action. Section 3 notes the definition of Latin rectangles, and general formulas specific to derangements are given in Section 4.  Sections 5 and 6 describe the poset diagrams for sizes four and five, respectively, while size six occupies sections 7 through 11.  In particular, sections 9 and 10 note basic hypergraph definitions and properties.  Finally, we note some basic connections to syzygies and representations of semi-magic squares in Section 12.

An unplanned by-product of this work is the enumeration of Latin squares as maximal chains (Corollary 12.2).  The initial motivation for this project was to find non-trivial applications of  Chu-Vandermonde convolution for finite graded posets; for each case, it is noted as a side comment, but in practice its repeated use was essential to stabilize numerical data and computations. 

For notation, we denote by $S_n$ the symmetric group on $\{1, \dots, n\}$ and by $D_{2n}$ the dihedral group with $2n$ elements.  Typically, when we identify a subgroup of dihedral type, we give  generating elements of order $n$ and $2$, leaving the reader to verify the defining relation $yxy^{-1}=x^{-1}.$

\section{Semi-magic squares and the poset $M(n, s)$}

\begin{definition}
An square matrix $M$ of size $n$ with non-negative integer entries is called a {\bf semi-magic square} with line sum $\rho(M)$ if the sum along any row or column equals $\rho(M).$   Let $M(n)$ denote the monoid of all semi-magic squares of size $n$ under addition.
\end{definition}
As implied by the definition, the set of semi-magic squares is closed under addition and multiplication by non-negative integers.  In fact, every linear combination of permutation matrices with non-negative integral weights is a semi-magic square, and, by the Birkhoff-von Neumann theorem, the converse also holds.

Denote by $G$ be the automorphism group $Aut(M(n))$ of the monoid of semi-magic squares of size $n$.  If $g$ is in $G$, we denote the action of $g$ on the semi-magic square $M$ by $g\cdot M$.  Then $g$ is a bijection on the set of semi-magic squares that respects the operation of addition and preserves the zero matrix:
$$g\cdot (M+N) = g\cdot M + g\cdot N, \qquad g\cdot 0 = 0.$$

In fact, we have

\begin{theorem}[\cite{LTT}, Theorem 2.2] Let $G$ be the group of automorphisms of $M(n).$  Then $G$ is isomorphic to the wreath product $S_n\wr \mathbb{Z}/2.$ This finite group, generated by  row and column permutations and transpose,  has order $2(n!)^2$.
\end{theorem}

 In particular, if $\sigma, \tau$ are in $S_n$ and $T$ represents the transpose operation, then the group elements and their corresponding action may be uniquely expressed by either
$$R(\sigma)C(\tau)\cdot M = P_\sigma M P_\tau^{-1},\quad \text{or}\quad R(\sigma)C(\tau)T\cdot M = P_\sigma M^T P_\tau^{-1};$$
the commuting elements $R(\sigma)$ and $C(\tau)$ represent row and column permutations, respectively, and non-commutativity is expressed by the relation
$$R(\sigma)T = TC(\sigma).$$

Define  $J$ to be the matrix with all entries equal to 1.  Note that $J$ is in $M(n)$ with $\rho(J) = n$.  The fixed points of the action of $G$ are precisely the multiples of $J.$

Now $M(n)$ admits a partial ordering using entry-wise comparison;  that is, $M\le N$ if $m_{ij} \le n_{ij}$ for all $1\le i, j \le n$.  For $s\ge 0,$ define
$$M(n, s) = \{M\in M(n)\ |\ M \le sJ\}.$$
With the induced partial order,   $M(n, s)$ is a finite graded poset with unique minimum $\hat{0}=0J$ and unique maximum $\hat{1}=sJ$ (for instance, \cite{StE} or \cite{StA}).  The rank function $\rho(M)$ is given by line sum, and $N$ covers $M$ if and only if $N=M+P_\sigma$ for some permutation matrix $P_\sigma$.  
Furthermore, since $M\le N$ implies $g\cdot M\le g\cdot N,$ the action of $G$ preserves $M(n, s).$

This poset is self-dual with involution
$$M' = sJ - M.$$
Now
$$\rho(M') = ns-\rho(M)$$
and, for $g$ in $G$, 
$$(g\cdot M)' = g\cdot(M').$$

Considering $M(n)$ as a subset of $\mathbb{Z}^{n^2}$, for which permutation matrices generate the lattice paths of interest, we denote the number of maximal chains between $\hat{0}$ and $M$ by the path number $v(M)$. We denote the number of elements in the  orbit corresponding to $M$ by $o_M.$ The path number $v(M)$ depends only on the orbit of $M$, as does covering data corresponding to $M.$

\section{Latin rectangles and the  poset $M(n, 1)$}
  
We now consider  the finite graded poset $M(n, 1)$.  These semi-magic squares are both $(0, 1)$-matrices and sums of permutation matrices.  That is, if we write these permutations in single line notation, the entries in columns are distinct when listed.  

\begin{definition}
Suppose $0\le m\le n.$  A $m\times n$ matrix $L$ with entries in $\{1, \dots, n\}$ is called a {\bf Latin rectangle} if each value occurs once in each row and at most once in each column.  If $m=n,$ we call $L$ a {\bf Latin square}.
\end{definition}

Thus we may first realize a path from $\hat{0}$ to $M$ in $M(n, 1)$ as a iterated sum of permutation matrices, and then transcribe this sequence into single line notation to obtain a Latin rectangle; the length of the path is the line sum of $M$, which is also the height of the corresponding Latin rectangle.  If the order of addition is recorded top-down, then every Latin square represents a distinct path from $\hat{0}$ to $J$.  For instance, in the poset $M(3, 1)$, we have

$$
\hspace{10pt}\hat{0}\hspace{30pt}\to\hspace{20pt}
\begin{matrix}
 2 & 3 &  1
\end{matrix}\quad\ \ \  \to\quad\ \ \
\begin{matrix}
2 & 3 &  1\\
1 & 2 &  3
\end{matrix}\quad \ \ \ \to\quad\ \ \
\begin{matrix}
2 & 3 & 1\\
1 & 2 &  3\\
3 & 1 &  2
\end{matrix}
$$

$$
\begin{bmatrix}
0 & 0 &  0\\
0 & 0 &  0\\
0 & 0 &  0
\end{bmatrix}
\quad\to\quad
\begin{bmatrix}
0 & 0 & 1\\
1 & 0 &  0\\
0 & 1 &  0
\end{bmatrix}\quad\to\quad
\begin{bmatrix}
1& 0 &  1\\
1 & 1 &  0\\
0 & 1 &  1
\end{bmatrix}\quad\to\quad
\begin{bmatrix}
1 & 1 & 1\\
1 & 1 &  1\\
1 & 1 &  1
\end{bmatrix}.
$$

Latin squares were named for the characters used in Euler's work.  The number of Latin squares of size $n$ (\cite{Slo}, A002860) remains an open problem, with exact values known up to size 11 at the time of this writing.  
See, for instance, \cite{DK} and \cite{LyM} for general theory, \cite{SW} for a general formula using permanents, and \cite{MKW} for an overview with computational results to size 11.  Our total Latin rectangle counts are verified by formula (1) and Figure 3 in \cite{Sto}; see also \cite{BLy}, \cite{Ge}, and \cite{Rio} for counting formulas for Latin rectangles.

\vspace{5pt}

\section{The second rank and derangements}

Our main goal is to construct the poset diagrams for $M(n, 1)$ for $4\le n \le 6.$  When $n=3,$ the full diagram may be displayed easily, and the diagram for orbits is a chain with four elements; see Figures 1 and 3 in \cite{Do}.  Otherwise our general goal consists of two parts:  determine the orbit data for each rank, and determine the covering data for the orbits. 

Under the group action, the orbits for ranks 0 and 1 are evident, and, by duality, so are the top two ranks.  For rank two, we first consider general sums of  pairs of permutation matrices to obtain path counts.

\begin{definition} We say the pair of distinct permutation matrices $P_{\sigma_1}, P_{\sigma_2}$  is {\bf uniquely summable} if, when
$P_{\sigma_1} + P_{\sigma_2} = P_{\tau_1} + P_{\tau_2},$
we have either $\sigma_1=\tau_1$,\ $\sigma_2=\tau_2$, or $\sigma_1=\tau_2$,\ $\sigma_2=\tau_1$.
\end{definition}

It is immediate that the notion of unique summability is invariant under $G$.  Next compare the following with Corollary 2.1 of \cite{BS}.

\begin{theorem} Suppose $\sigma_1^{-1}\sigma_2$ is a product of $c$ disjoint cycles in $S_n.$
Then  $P_{\sigma_1}, P_{\sigma_2}$ is a uniquely summable pair if and only if $c=1.$ Furthermore, the matrix $P_{\sigma_1}+P_{\sigma_2}$ may be expressed in $2^{c}$ ways as an ordered sum of permutation matrices.
\end{theorem}

\begin{proof} Since left translation preserves pairwise sums of permutation matrices, we may assume $P_{\sigma_1} = I$.  
By conjugation, we may assume $\sigma_2$ is represented as a product of disjoint cycles with first cycle $(1\dots i)$ and fixed points at the end.   Now $M=I+P_{\sigma_2}$ is block diagonal, with blocks corresponding to cycles, followed by diagonal entries of two for the fixed points.   

If we rewrite $M=P_{\tau_1} + P_{\tau_2},$ then $\tau_1(1)$ equals 1 or $2,$ and the line sum of 2 determines all other values for $\tau_1$ on $\{2, \dots, i\}$.   That is, $\tau_1$ is either the identity or $(1\dots i)$ on $\{1, \dots, i\}$.  
Generalizing, we have shown that cycles may exchange and that only products of cycles may exchange. Hence the count follows.
\end{proof}
 Restated in terms of path numbers, we have
 \begin{corollary} Suppose $\sigma^{-1}\tau$ is a product of $c$ disjoint cycles.  Then there are $2^{c}$ paths from $\hat{0}$ to $P_\sigma+P_\tau$ in $M(n)$.
\end{corollary} 
 
 Specializing these sums to respect the maximum property, we now characterize the orbit types in rank two.
 
 \begin{theorem}
 The orbits in the second rank of  $M(n, 1)$ are in one-one correspondence with the conjugacy classes of derangements of $\{1, \dots, n\}$.
\end{theorem}
\begin{proof}
Recall that $\sigma$ is a derangement if it has no fixed points.  If $M = P_{\tau_1} + P_{\tau_2}$ is in $M(n, 1)$, left-multiplying by $P_{\tau_1^{-1}}$ yields $I + P_\sigma$ for some $\sigma$; the derangement property of $\sigma$ follows since the maximum entry is still one. 

Next, conjugations by elements in $S_n$ fix the identity and act transitively on permutations with the same cycle structure.  Thus each orbit in rank two corresponds to at least one derangement class. Furthermore, $\sigma$ and $\sigma^{-1}$ have the same cycle structure, so orbits can be determined without using transpose.  Finally, since we only need to calculate up to conjugacy, it is enough to consider only left multiplication to show the derangement class is unique.

Suppose $P_\sigma\cdot(I+P_{\tau_1}) = I+P_{\tau_2}$, where the $\tau_i$ are derangements.  Since $\tau_2$ is a derangement, the fixed points of $\sigma$ and $\sigma\tau_1$ partition $\{1, \dots, n\}$, and $\tau_2$ is the disjoint product of the cycles in $\sigma$ and $\sigma\tau_1.$  When the partition is trivial, we have $\sigma=e$ or $\sigma=\tau_1^{-1}$, and the $\tau_i$ belong to the same class.

If each element has at least one fixed point, then we may relabel the indices by conjugating, so that the fixed points of $\sigma\tau_1$ are $\{1, \dots, k\}.$ Since $\sigma=\tau_2$ on $\{1,\dots, k\},$  $\tau_1=\tau_2^{-1}$ there by the equation. Thus the corresponding cycles have the same structure.  Likewise, on the other set of fixed points, we have $\tau_1=\tau_2.$\end{proof}

Next, the size of an orbit may be computed using the Orbit-Stabilizer Theorem. The following theorem and corollaries handle all cases needed here.

\begin{theorem}  Suppose $n>2,$ and let $\sigma = (1\dots n).$ The stabilizer of $M=I+P_{\sigma}$ is isomorphic to $D_{4n},$ the dihedral group with $4n$ elements.  The corresponding orbit has $o_M =  (n-1)!n!/2$ elements.
\end{theorem}
\begin{proof}
As a cycle, $\sigma$  is centralized by precisely its own powers in $S_n$, and, as a permutation matrix, further stabilized by 
$R(\tau)C(\tau)T$ with $\tau = (1n)(2\ n-1)\dots.$  
This stabilizer element corresponds to reflection across the counter-diagonal; after transposing, we rotate the matrix by 180 degrees.  The cyclic subgroup is in fact a subgroup of the cyclic subgroup of order $2n$ generated by $R(n\dots 1)T,$ and these two  elements with factors of $T$ generate a stabilizing subgroup isomorphic to $D_{4n}$.

The stabilizer contains no other elements.  Consider the one in the upper-right corner of $I+P_\sigma$.  Under a stabilizing element, there are $2n$ choices to replace it.  If we suppose this one is fixed, then the elements in the same row and column are either fixed or interchange. Since the line sum is 2, the positions of the remaining ones are fixed by this choice, so the symmetry is the identity or counter-transpose.  Thus the order of the stabilizer is at most $4n$.
\end{proof}

These symmetries may also be modeled by the isometries of a bounder right cylinder over a regular $n$-gon. Effectively, every symmetry above is composed of at most three operations: translation along the main diagonal, interchange of diagonals of ones, and counter-transpose.  For a general derangement, each cycle of size $k$ generates a stabilizing subgroup of itself of size $k$; these correspond to translations along the main diagonal within a block. 

\begin{corollary}
Suppose the derangement $\sigma$ is a product of two disjoint cycles $\sigma_1$ and $\sigma_2$.  Let $M=I + P_\sigma.$

(a)  If $\sigma_1$ and $\sigma_2$ have lengths $n_1 > n_2 \ge 2$, then the stabilizer of $M$ has order $8n_1n_2.$   The corresponding orbit has $o_M = (n!)^2/4n_1n_2$ elements. 

(b) If $n> 4$ and $n_1=n_2$, the stabilizer of $M$ has order $4n^2$, and the corresponding orbit has $o_M= (n-1)!^2/2$ elements.
\end{corollary}
\begin{proof} We may assume the cycles have consecutive indexing.  For part (a), first suppose $n_1=2.$ Since the block of size 2 is preserved by transpose, all stabilizing symmetries of the larger block occur.  In turn, this subgroup normalizes the four element subgroup generated by transpositions on the block of size 2. 

If $n_1 > 2$, there are at most $16n_1n_2$ symmetries by Theorem 4.5. Furthermore, the actions on each block either both use transpose or neither do. From the proof of Theorem 4.5, there are $4n_1n_2$ elements of each types. Since this subgroup is a proper subgroup of the direct product, the result follows.

For part (b),  a similar argument holds, but now the subgroup is further normalized by counter-transpose, which switches blocks.
\end{proof}

Next, we have
\begin{corollary} Suppose $n=2m$ and the derangement $\sigma$ is an involution. Then the stabilizer of $M=I+P_\sigma$ has order $m! 2^{n+1},$ and the corresponding orbit has $(n!)^2/m!2^{n}$ elements.
\end{corollary}
\begin{proof}
As before, we assume consecutive indexing of cycles in $\sigma.$  The subgroup that preserves blocks has a $(\mathbb{Z}/2\times \mathbb{Z}/2)^m$ subgroup generated by transpositions as row and column switches, which is further normalized by $T$.  It is then straightforward to find a subgroup isomorphic to $S_m$ that permutes the blocks of $M$. For instance, a transposition that exchanges blocks $k$ and $k+1$ on the main diagonal is given by $R(\tau)C(\tau)$, where $\tau=(2k-1\ 2k+1)(2k\ 2k+2).$
\end{proof}

Finally, we give the general formula for the path numbers in rank $n-1$ and $n$, assuming all path numbers in rank $n-2$ are known.  Implicit in path counting for higher ranks is the use of an order-raising operator (or ``up" operator) for graded posets (for instance,  \cite{StA}); generalizing Pascal's identity, the path number at a given element  $M$ is the sum of the path numbers of all elements covered by $M$.

\begin{proposition}
Suppose $M = J-I$ in $M(n, 1),$ and  let $M_\sigma = J-I-P_{\sigma}$, where $\{\sigma\}$ is a set of representatives for each class of derangements in $S_n.$ Then the path number for $M$ is given by
$$v(M) = \sum\limits_\sigma\ c_\sigma\cdot v(M_\sigma),$$
where $c_\sigma$ is the number of elements in the conjugacy class for $\sigma.$  Additionally, $$v(J)=n!\cdot v(J-I),$$
which also equals the number of Latin squares of size $n.$
\end{proposition}

\begin{proof}  Since there are non-nonzero diagonal elements,  the only contributions to the path number of $M$ occur from derangements, and each element of a derangement class contributes the same path number by homogeneity.  Likewise, the second statement follows since each permutation matrix contributes exactly once to $v(J).$
\end{proof}

\section{The poset $M(4, 1)$}

\begin{figure}[ht]
{\tiny
 \begin{tikzpicture}
   \node (e) at (3, 4) {$J$};
  \node (d) at (3,  3) {$J-I$};
  \node (c1) at (4, 2) {$I+P_{(12)(34)}$};
  \node (c) at (2, 2) {$I+P_{(1234)}$};
  \node (b) at (3, 1)  {${P_\sigma}$};
  \node (a) at (3,0) {$0$};
  \draw  (a) edge (b) (b) edge (c) (c) edge (d)  (d) edge (e) (b) edge (c1) (c1) edge (d);
\end{tikzpicture}
}\qquad
{ \tiny
\begin{tikzpicture}
   \node (e) at (3, 4) {${1,576}$};
  \node (d) at (3, 3) {${24, 24}$};
  \node (c1) at (4, 2) {${18, 4}$};
  \node (c) at (2, 2) {${72, 2}$};
  \node (b) at (3,1)  {${24, 1}$};
  \node (a) at (3,0) {${1,1}$};
  \draw  (a) edge (b) (b) edge (c) (c) edge (d)  (d) edge (e) (b) edge (c1) (c1) edge (d);
\end{tikzpicture}
}
\caption{Poset diagrams for orbits in $M(4, 1)$}
\end{figure}
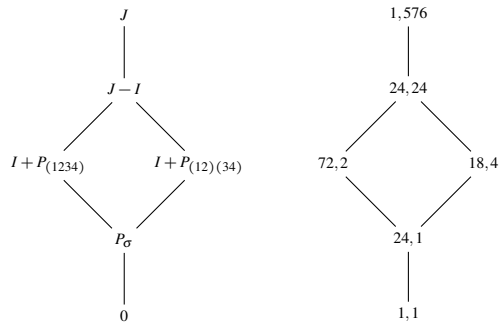

Consider the case when $n=4,$ for which Figure 1 gives the poset diagrams of orbits.  The numbers on the right give the corresponding orbit size and path number.  In this case, all of these numbers are addressed in Section 4.  We confirm that $v(J)= 576$, the number of Latin squares of size 4.

Although the path number for rank 3 follows from Proposition 4.8  with 
$$c_{(1234)}=6,\quad c_{(12)(34)}=3,$$ the entirety of the underlying process to be adapted may be seen with maximum clarity here.  We fix the element $M=J-P_{(13)(24)}$ and consider all elements covered by it.  These elements correspond to derangements of type $(abcd)$ and $(ab)(cd)$, with path numbers 2 and 4, respectively.  Once enumerated, we add the path numbers.

First we note how to extract a permutation from $M$.  For $M$ it is possible to subtract the permutation $\sigma = (12)(34)$, given in single line notation $2134.$  The extraction process looks like the following:
$$\begin{bmatrix} 1 & 1 & 0 & 1\\ 1 & 1 & 1 & 0\\ 0 & 1 & 1 & 1\\ 1 & 0 & 1 & 1\end{bmatrix}\ \ \to\ \  
\begin{matrix}
1 & 2 &  \fbox{3} & \fbox{4}\\
\fbox{2} & 3 &  4 & 1\\
4 & \fbox{1} & 2 & 3\\
\end{matrix}\ \ \to\ \  
\begin{matrix}
1 & 2 &  4 & 3 \\
4 & 3 &  2 & 1 \\
\end{matrix}\ \ \to\ \ 
\begin{matrix}
1 & 2 &  3 & 4\\
4 & 3 &  1 & 2\\
\end{matrix}
$$
The second entry, the Latin rectangle for $M$,  records the row positions of the ones in the corresponding  column, arranged to reflect the Latin rectangle property. The third entry is the Latin rectangle with the string $2134$ removed, each digit corresponding to the same column. Finally, we rearrange columns to obtain the string $1234$ in the first row, giving a derangement in the second row. In any case, the covered element is $I  + P_{(1423)}.$  

We enumerate all such permutations with a tree, in which each level corresponds to a column of the Latin rectangle and each branch corresponds to a permutation of $\{1, 2, 3, 4\}$ to be extracted.  The tree for $M$ is given in Figure 2, with $M$ covering 6 elements with path number 2 and 3 elements with path number 4. Thus the path number of $M$ is 24.

Finally, one verifies the poset convolution formula for orbits in this case (\cite{Do}, Section 6):
$$v(J)\ \ =\ \ \sum\limits_{\rho(M)=k} o_M \cdot v(M)\cdot v(J-M)\ \ =\ \  1 \cdot 24\cdot 24\  =\  72\cdot 2\cdot 2 + 18 \cdot 4 \cdot 4\  =\  576.$$
Here $0\le k \le 4$ is fixed,  $M$ ranges over a set of representatives for each orbit in rank $k$, and $o_M$ denotes the number of elements in the orbit for $M$.

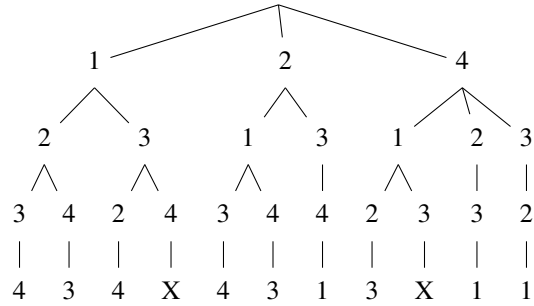
\begin{figure}
\begin{forest}
for tree={
  l sep=10pt,
  parent anchor=south,
  align=center
}
[
  [1\\
    [2\\
      [3\\
        [4\\
        ]        
      ]
      [4\\
        [3\\
        ]
      ]
    ]
    [3\\
      [2\\
        [4\\
        ]
      ]
      [4\\
        [X\\
        ]
      ]
    ]
  ]
  [2\\
    [1\\
      [3\\
        [4\\
        ]        
      ]
      [4\\
        [3\\
        ]
      ]
    ]
    [3\\
      [4\\
        [1\\
        ]
      ]      
    ]
  ]  
  [4\\
    [1\\
      [2\\
        [3\\
        ]        
      ]
      [3\\
        [X\\
        ]
      ]
    ]
    [2\\
      [3\\
          [1\\
          ]
        ]     
       ]
       [3\\
         [2\\
           [1\\
           ]
          ]
         ]  
  ]
]
\end{forest}
\caption{Downward increments for $J-P_{(13)(24)}$ in $M(4, 1)$}
\end{figure}

\begin{figure}[ht]
{\tiny
 \begin{tikzpicture}
  \node (e) at (3, 9) {$J$};
  \node (d2) at (3,  8) {$P_\sigma'$};
  \node (d1) at (1.5,  7) {$J-I-P_{(12345)}$};
  \node (d) at (4.5,  7) {$J-I-P_{(12)(345)}$};
  \node (c1) at (1.5, 6) {$I+P_{(12)(345)}$};
  \node (c) at (4.5, 6) {$I+P_{(12345)}$};
  \node (b) at (3, 5)  {${P_\sigma}$};
  \node (a) at (3,4) {$0$};
  \draw  (a) edge (b) (b) edge (c) (c) edge (d) (c) edge (d1) (d2) edge (e) (b) edge (c1) (d1) edge (d2) (d) edge (d2) (c1) edge (d1);
\end{tikzpicture}
}\qquad
{ \tiny
\begin{tikzpicture}
  \node (e) at (3, 9) { ${1,161280}$};
  \node (d2) at (3,8) {${120, 1344}$};
  \node (d1) at (4.5, 7) {${600, 24}$};
  \node (d) at (1.5, 7) {${1440, 36}$};
  \node (c1) at (1.5, 6) {${600, 4}$};
  \node (c) at (4.5, 6) {${1440, 2}$};
  \node (b) at (3,5)  {${120, 1}$};
  \node (a) at (3,4) {${1,1}$};
 \draw  (a) edge (b) (b) edge (c) (c) edge (d) (d2) edge (e) (b) edge (c1) (d1) edge (d2) (c) edge (d1) (d) edge (d2) (c1) edge (d);
\end{tikzpicture}
}
\caption{Poset diagrams for orbits in $M(5, 1)$}
\end{figure}
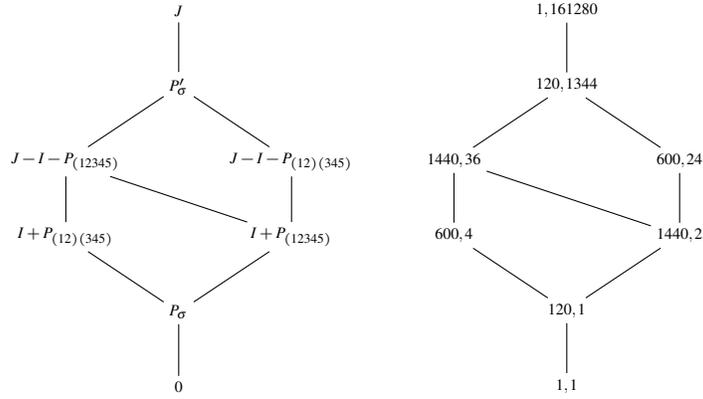

\section{The poset $M(5, 1)$}

For $n=5,$ we proceed in a similar manner to construct the poset diagrams in Figure 3.  In this case, at issue are the path numbers in rank 3.  Here we proceed as in the previous section. 

For  rank 3, the class of a covered element in rank 2 is quickly determined by checking for columns with the same entries.  When   $M=J-I-P_{(12345)}$, $M$ covers 8 elements with path number 2 and 5 elements with path number 4, giving $M$ a path number of 36.  For the other class, $M$ covers exactly 12 elements with path number 2, giving a path number of 24.

Finally, the convolution formula verifies the Latin square count:
$$161,280 = 120\cdot 1344 = 600\cdot 4\cdot 24 + 1440 \cdot 2 \cdot 36.$$

\section{The poset $M(6, 1)$}

\begin{figure}
{\tiny
 \begin{tikzpicture}
  \node (g) at (4.5, 7.8) {$[J]$};
  \node (f) at (4.5,  6.5) {$[P_\sigma']$};
  \node (e1) at (1, 5.2) {$[A']$};
  \node (e2) at (3.5, 5.2) {$[B']$};
  \node (e3) at (5.5, 5.2) {$[C']$};
  \node (e4) at (8, 5.2) {$[D']$};
  \node (d1) at (5.5,  3.9) {$[IV]$};
    \node (d2) at (7.25,  3.9) {$[V]$};
      \node (d3) at (1.75,  3.9) {$[II]$};
        \node (d4) at (9,  3.9) {$[VI]$};
          \node (d5) at (0,  3.9) {$[I]$};
            \node (d6) at (3.5,  3.9) {$[III]$};
  \node (c1) at (1, 2.6) {$[A]$};
  \node (c2) at (3.5, 2.6) {$[B]$};
  \node (c3) at (5.5, 2.6) {$[C]$};
  \node (c4) at (8, 2.6) {$[D]$};
  \node (b) at (4.5, 1.3)  {${[P_\sigma]}$};
  \node (a) at (4.5,0) {$[0]$};
  \draw (a) edge (b) (b) edge (c1) (b) edge (c2) (b) edge (c3) (b) edge (c4)
  (e1) edge (f) (e2) edge (f) (e3) edge (f) (e4) edge (f) (f) edge (g)
  (d1) edge (e1) (d1) edge (e2) (d1) edge (e3) (d1) edge (e4) (d2) edge (e3) (d3) edge (e3) (d3) edge (e1) (d3) edge (e2) (d3) edge (e3) (d4) edge (e2) (d4) edge (e3) (d4) edge (e4) (d5) edge (e1) (d5) edge (e2) (d6) edge (e2) (d6) edge (e3)
  (c1) edge (d1)  (c1) edge (d3)  (c1) edge (d5) (c2) edge (d1)  (c2) edge (d3)  (c2) edge (d4) (c2) edge (d5)  (c2) edge (d6)  
  (c4) edge (d1)  (c4) edge (d4)  (c3) edge (d1)  (c3) edge (d2) (c3) edge (d3)  (c3) edge (d4)  (c3) edge (d6) ;
  \end{tikzpicture}
}
\caption{Poset diagram for orbits in $M(6, 1)$}
\end{figure}
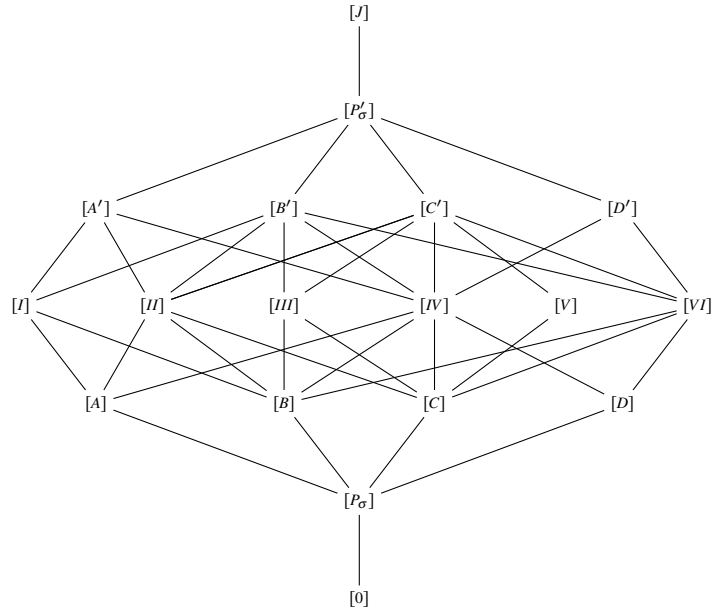

Finally we construct the poset diagram for $M(6, 1)$ (Figure 4) and verify that the number of Latin squares of size 6 is 812,851,200.  Description of the classes for each label appear in the following sections.  The general procedures for computing orbit data and path numbers are as in previous sections, but now we introduce hypergraphs to describe the classes in rank 3.

The entirety of the poset data is summarized in Table 1, with specific covering data listed in Tables 2 and 3.  For Tables 2 and 3, the entries denote the number of elements by type covered (column) by a given element (row).  Since the path number is the sum of all path numbers covered, we obtain the total path number for an entry in a given row by taking the dot product of the given row (multiplicities of type) with the top row (path numbers of covered types).
The last column of Table 3 is the top line of Table 2.

Finally the verification of the poset convolution formula is 

\begin{align*} 
812,851,200 & =  16200\cdot 4\cdot 4224 + 43200\cdot  2\cdot 4032 + 7200\cdot 4\cdot 4608 + 1350\cdot 8\cdot 5376\\
& = 86400\cdot 48^2 +  129600\cdot 48^2 + 16200\cdot 48^2 + 43200\cdot 72^2\\
&\qquad \qquad  + 200\cdot 144^2 + 21600\cdot 48^2
\end{align*}

\begin{table}
\caption{Orbit data for $M(6, 1)$}
\label{tab:1}
\begin{tabular}{|p{1cm}|p{1cm}|p{1cm}|p{1cm}||p{1cm}|p{1cm}|p{1cm}|p{1.7cm}|}
\hline\noalign{}
 $M$ & $\rho(M)$ & $o_M$ & $v(M)$ & $M$ & $\rho(M)$ & $o_M$ & $v(M)$\\
 \noalign{}\hline\noalign{}
$\hat{0}$ & 0 & 1 & 1 &  $P_\sigma$ & 1 & 720 &  1\\
 $A$ & 2 &  16200 & 4 & $I$ & 3 & 86400 & 48\\
 $B$ & 2 & 43200 & 2 & $II$ & 3 & 129600 & 48 \\
 $C$ & 2 & 7200 & 4  & $III$ & 3 & 16200 & 48\\
 $D$ & 2 & 1350 & 8 & $IV$ & 3 & 43200 & 72 \\
 $A'$ & 4 & 16200 & 4224 & $V$  & 3 & 200 & 144\\
 $B'$ & 4 & 43200 & 4032 & $VI$ & 3 & 21600 & 48 \\
 $C'$ & 4 & 7200 & 4608 & $P_\sigma'$ & 5 & 720 & 1128960\\
 $D'$ & 4 & 1350 & 5376 &   $J$ & 6 & 1 & 812851200 \\
\noalign{}\hline\noalign{}
\end{tabular}
\end{table}

\section{Orbit sizes for the second rank}

First we indicate the indexing for orbit representatives of type $I+P_\sigma$, where $\sigma$ is given by
$$A:\  (12)(3456),\qquad B:\  (123456),\qquad C:\  (123)(456),\qquad D:\  (12)(34)(56),$$
and, in rank 4, orbit representatives are given by, for instance, $A'=J-I-P_\sigma.$  Path numbers in rank 2 are given by Corollary 4.3. To use Proposition 4.8, the conjugacy classes of derangements have orders
$$A:\  90,\qquad B:\  120,\qquad C:\  40,\qquad D:\  15.$$

To determine orbit sizes, again we apply the results of Section 4. Classes $A$ and $C$ are given by Corollary 4.6, while Theorem 4.5 calculates class $B.$   Type $D$ follows from Corollary 4.7.  The rank 4 class sizes now follow by duality.

\begin{table}
\caption{Covering relations for rank 4 over rank 3}
\label{tab:1}
\begin{tabular}{|p{1.5cm}|p{1cm}|p{1cm}|p{1cm}|p{1cm}|p{1cm}|p{1cm}|p{1cm}|}
\hline\noalign{}
 Rank $4/3$ & $I$ & $II$ & $III$ & $IV$ & $V$ & $VI$ & $v(M)$\\
 $v(M)$ &  48 & 48 & 48 & 72 & 144 & 48 & \\
 \noalign{}\hline\hline\noalign{}
 $A'$ & 32  & 32 & 0  & 16 & 0 & 0 & 4224 \\
 $B'$ & 24  & 36 & 6 & 8 & 0 & 6 & 4032 \\
 $C'$ & 0 &  36 & 9 & 24 & 1 & 12 & 4608 \\
 $D'$ & 0 & 0  & 0 & 64 & 0 & 16  & 5376 \\
\noalign{}\hline\noalign{}
\end{tabular}
\end{table}

\begin{table}
\caption{Covering relations for rank 3 over rank 2}
\label{tab:1}
\begin{tabular}{|p{1.5cm}|p{1cm}|p{1cm}|p{1cm}|p{1cm}|p{1cm}|p{1cm}|p{1cm}|}
\hline\noalign{}
 Rank  $3/2$& $A$ & $B$ & $C$ & $D$ & $v(M)$\\
 $v(M)$ &  4 & 2 & 4 & 8 & \\
 \noalign{}\hline\hline\noalign{}
 $I$ & 6  & 12 & 0  & 0  & 48 \\
 $II$ & 4  & 12 & 2 & 0 & 48 \\
 $III$ & 0 &  16 & 4 & 0 & 48 \\
 $IV$ & 6 & 8  & 4 & 2 &  72 \\
  $V$ & 0 &  0 & 36 & 0 & 144 \\
 $VI$ & 0 & 12  & 4 & 1 & 48 \\
\noalign{}\hline\noalign{}
\end{tabular}
\end{table}

\section{Basic notions for hypergraphs}

In the third rank, the new issue is how to both determine and distinguish the orbits, and compute stabilizers of representative elements.  

We review the basic notions of hypergraphs. See, for instance, \cite{Bol} for general theory, but here we need little more than basic definitions.

A {\bf hypergraph} $H$ is a pair of  sets $(X, E)$, where the elements of $X$ are called the {\bf vertices} of $H$ and the elements of $E$, called {\bf hyperedges}, are subsets of $X.$ We suppose  $H$ has both finitely many vertices $x_1, \dots, x_n$ and  hyperedges $e_1, \dots, e_m$.  This definition allows for repeated hyperedges.

\begin{definition}
The {\bf incidence matrix} $M$ of $H$ is the $n\times m$ matrix with non-zero entries $m_{ij}=1$ when $x_i \in e_j.$ That is, the columns of $M$ record the vertices in each hyperedge. 
\end{definition}

\begin{definition}  Let $H$ be a hypergraph on $n$ vertices.

(a) We call $H$  {\bf $k$-uniform} if every hyperedge contains $k$-vertices.  That is, the incidence matrix has exactly $k$ ones in each column.

(b)  We call $H$ {\bf $k$-regular} if every vertex is contained in $k$ hyperedges. That is, the incidence matrix has exactly $k$ ones in each row.

(c)  Suppose $0\le k\le n.$ We call $H$ {\bf semi-magic} (of rank $k$) if it is both $k$-uniform and $k$-regular.  That is, the incidence matrix of $H$ is an element of $M(n, 1)$ with line sum $k.$
\end{definition} 
 A semi-magic hypergraph of rank 2 is a union of cycles without isolated vertices, with an obvious correspondence to derangements. For rank 3, we have an arrangement of $n$ triangles in an regular $n$-gon, where exactly  three triangles abut each vertex.

\begin{definition}
The {\bf dual} of the hypergraph $H$, denoted $H^*$, is the hypergraph with incidence matrix $M^T$.
\end{definition}
The properties of $k$-uniformity and $k$-regularity interchange under duality.  Of course, the dual of a semi-magic hypergraph of rank $k$ is also semi-magic of rank $k$.  In this case, it will be convenient to assume that $H$ and $H^*$ have the same sets of vertices.

In a similar manner, we can define complements.  
\begin{definition}
 The {\bf complement} of the hypergraph $H$, denoted by $H'$, is the hypergraph with incidence matrix $M'=J-M.$  That is, the hyperedges of $H'$ are given by the set complements $e_i'=X \setminus e_i.$
 \end{definition}
 
If $H$ is semi-magic of rank $k$, then $H'$ is semi-magic of rank $n-k$, and  complementation of semi-magic hypergraphs corresponds to duality in $M(n, 1).$

Finally, we note when two hypergraphs are the same up to indexing.
\begin{definition}
Two hypergraphs $H_1$, $H_2$ with vertex set $X$ are said to be {\bf equivalent} if there exists a bijection on $X$ that induces bijections between the hyperedges of $H_1$ and $H_2$.   

The set of all equivalences of $H$ with itself is called the {\bf automorphism group} of $H$, denoted by $G_H$.  
\end{definition}

\section{The action of $G$ on hypergraphs}

With the language of hypergraphs in place, we characterize the orbits of $M(n, 1)$ in rank $k$ with respect to the notion of equivalence of semi-magic hypergraphs of rank $k.$   The action of $G$ on $M(n, 1)$ induces an action on hypergraphs by way of the incidence matrix $M$.   While this action applies to any hypergraph with a square incidence matrix, we assume $M$ is in $M(n, 1)$, so the hypergraph is semi-magic.  While the hypergraph is entirely determined by the incidence matrix, we will find it useful to consider hypergraph pairs generated by the row and column vectors.

\begin{definition}
Suppose $M$ is in $M(n, 1)$ and $H$ is the hypergraph associated to $M.$ We call $(H, H^*)$ the {\bf hypergraph pair} associated to $M.$
\end{definition}

The following proposition characterizes the equivalence of hypergraphs entirely as row and column switches of $M$.  In the next section, we will see an example where $H$ and $H^*$ are not equivalent.
\begin{proposition}  Suppose the hypergraph pair $(H, H^*)$ is associated to $M$.  Generators of $G$ act as follows
\begin{enumerate}
\item[$R(\sigma)$]:  $\sigma$ induces a bijection on $X$ with respect to $H$, and  the edges of $H^*$ are permuted,
\item[$C(\tau)$]:  $\tau$ induces a bijection of $X$ with respect to $H^*$, and the edges of $H$ are permuted, and
\item[$T$]: transpose interchanges $H$ and $H^*$.
\end{enumerate}
That is, each orbit of $G$ on $M(n, 1)$ in rank $k$ corresponds to an unordered pair of equivalence classes of semi-magic hypergraphs of rank $k$ on $n$ vertices.
\end{proposition}

Now the stabilizer of $M$ may be reduced to consideration of automorphism groups.

\begin{proposition} 

Let $(H, H^*)$ be the hypergraph pair associated to $M$.  

(a)  $G_H$ and $G_{H^*}$ are isomorphic.

(b)  If $H$ and $H^*$ are inequivalent, then the stabilizer of $M$ is isomorphic to $G_H$.

(b)  Otherwise, the stabilizer of $M$ contains $G_H$ as a normal subgroup of index 2. 
\end{proposition}

\begin{proof} 
For part (a), the isomorphism is given by $g\mapsto TgT$, or $R(\sigma)C(\tau)\mapsto R(\tau)C(\sigma).$  For part (b), inequivalence means no element in the stabilizer exchanges $H$ and $H^*$, so the elements cannot have a factor of $T$.  

On the other hand, if $H$ and $H^*$ are equivalent, $R(\sigma)C(\tau)\cdot M = T\cdot M$ for some $\sigma$ and $\tau$, and thus there  exists an element $T_1$ in the stabilizer of $M$ that maps $H$ to $H^*$.  Since  elements of $G_H$ have no factor of $T$, $G_H$ is normalized by $T_1$.  Finally, the index assertion follows since $T_1^2$ is in $G_H$.
\end{proof}

\begin{figure}
$$Ia:\ \ \begin{bmatrix}
1 & 1 &  0 & 0 & 1 & 0\\
1 & 1 &  0 & 1 & 0 & 0\\
0 & 0 & 1 & 1 & 1 & 0\\
0 & 0 & 1 & 1 & 0 & 1\\
0 & 0 & 1 & 0 & 1 & 1\\
1 & 1 & 0 & 0 & 0 & 1\\
\end{bmatrix}\quad \rightarrow\quad 
\begin{matrix}
1 & 2 &  3 & 4 & 5 & 6\\
2 & 6 &  4 & 3 & 1 & 5\\
6 & 1 & 5 & 2 & 3 & 4\\
\end{matrix}\quad \rightarrow\quad\hspace{-125pt}
\vcenter{\includegraphics[scale=.6]{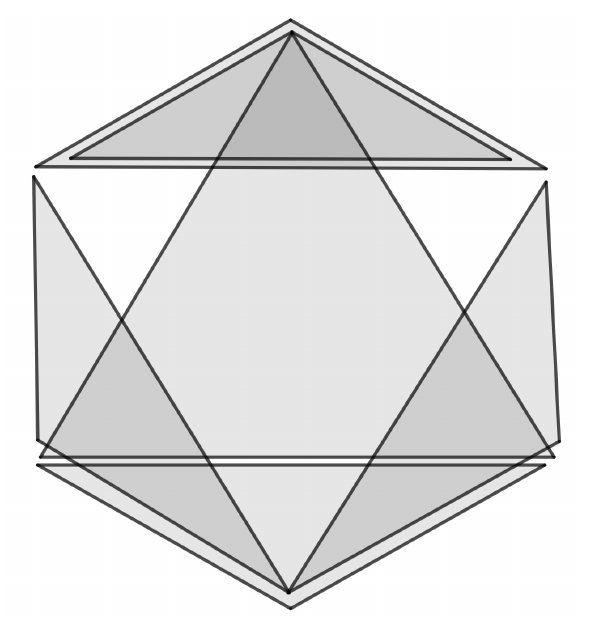}}$$
$$Ib:\ \ \begin{bmatrix}
1 & 0 &  0 & 0 & 1 & 1\\
1 & 1 &  1 & 0 & 0 & 0\\
0 & 1 & 1 & 1 & 0 & 0\\
0 & 0 & 1 & 1 & 0 & 1\\
0 & 1 & 0 & 1 & 1 & 0\\
1 & 0 & 0 & 0 & 1 & 1\\
\end{bmatrix}\quad \rightarrow\quad 
\begin{matrix}
1 & 2 &  3 & 4 & 5 & 6\\
2 & 5 &  4 & 3 & 6 & 1\\
6 & 3 & 2 & 5 & 1 & 4\\
\end{matrix}\quad \rightarrow\quad\hspace{-125pt}
\vcenter{\includegraphics[scale=.6]{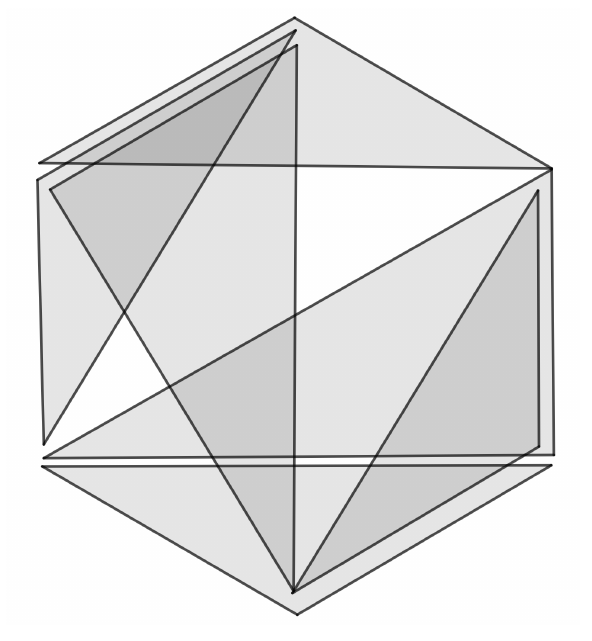}}$$
$$II:\ \ \ \begin{bmatrix}
1 & 1 &  0 & 0 & 0 & 1\\
1 & 1 &  0 & 0 & 1 & 0\\
1 & 0 & 1 & 1 & 0 & 0\\
0 & 0 & 1 & 1 & 1 & 0\\
0 & 0 & 0 & 1 & 1 & 1\\
0 & 1 & 1 & 0 & 0 & 1\\
\end{bmatrix}\quad \rightarrow\quad 
\begin{matrix}
1 & 2 &  3 & 4 & 5 & 6\\
2 & 1 &  6 & 3 & 4 & 5\\
3 & 6 & 4 & 5 & 2 & 1\\
\end{matrix}\quad \rightarrow\quad\hspace{-125pt}
\vcenter{\includegraphics[scale=.6]{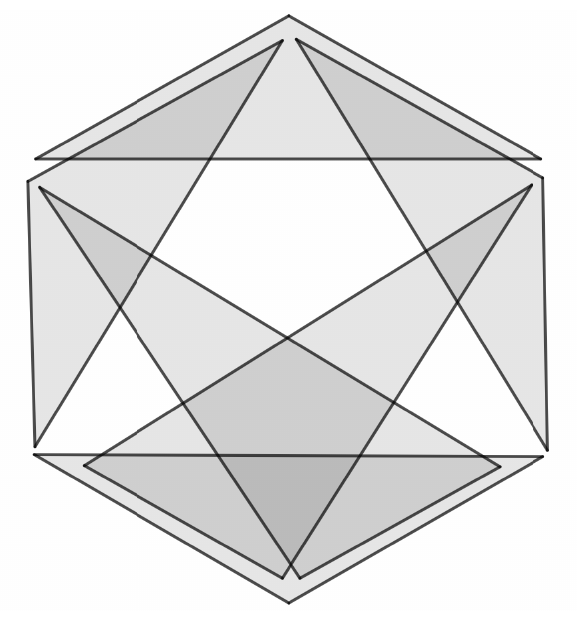}}$$
$$III:\ \ \begin{bmatrix}
1 & 1 &  1 & 0 & 0 & 0\\
1 & 1 &  0 & 0 & 0 & 1\\
0 & 0 & 1 &  1 & 1 & 0\\
0 & 0 & 0 & 1 & 1 & 1\\
0 & 0 & 1 &  1 & 1 & 0\\
1 & 1 & 0 &  0 & 0 &1 \\
\end{bmatrix}\quad \rightarrow\quad 
\begin{matrix}
1 & 2 &  3 & 4 & 5 & 6\\
2 & 6 &  1 & 5 & 3 & 4\\
6 & 1 & 5 & 3 & 4 & 2\\
\end{matrix}\quad \rightarrow\quad\hspace{-125pt}
\vcenter{\includegraphics[scale=.6]{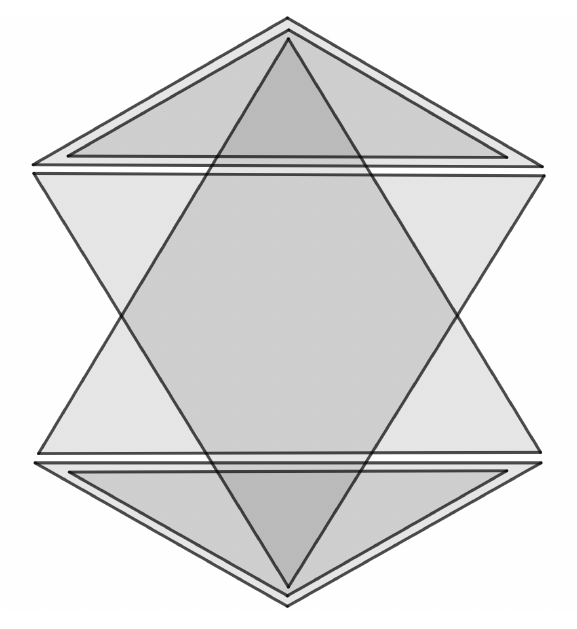}}$$
\caption{Orbits with rank $3$:\ cases $I-III$}
\end{figure}

\begin{figure}

$$IV: \ \ \begin{bmatrix}
1 & 1 &  0 & 0 & 0 & 1\\
1 & 1 &  1& 0 & 0 & 0\\
0 & 1 & 1 & 1 & 0 & 0\\
0 & 0 & 1 & 1 & 1 & 0\\
0 & 0 & 0 & 1 & 1 & 1\\
1 & 0 & 0 & 0 & 1 & 1\\
\end{bmatrix}\quad \rightarrow\quad 
\begin{matrix}
1 & 2 &  3 & 4 & 5 & 6\\
2 & 3 &  4 & 5 & 6 & 1\\
6 & 1 & 2 & 3 & 4 & 5\\
\end{matrix}\quad \rightarrow\quad\hspace{-125pt}
\vcenter{\includegraphics[scale=.6]{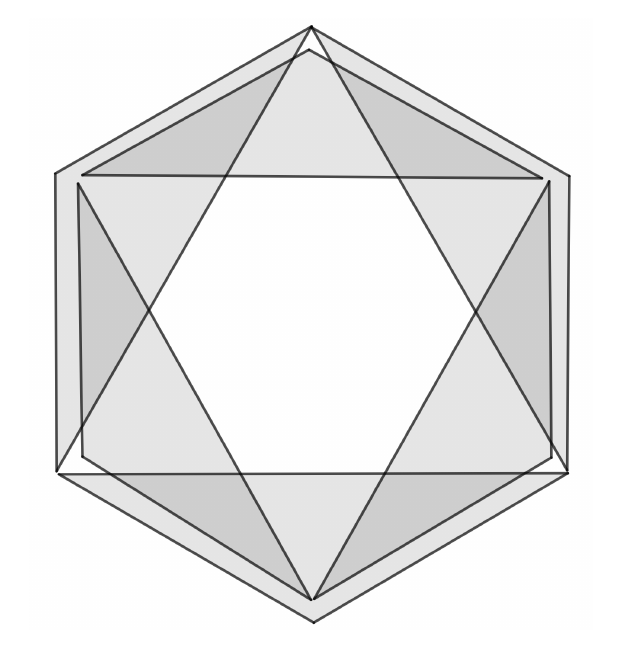}}$$
$$V:\ \ \ \ \begin{bmatrix}
1 & 1 &  1 & 0 & 0 & 0\\
1 & 1 &  1& 0 & 0 & 0\\
1 & 1 & 1 & 0 & 0 & 0\\
0 & 0 & 0 & 1 & 1 & 1\\
0 & 0 & 0 & 1 & 1 & 1\\
0 & 0 & 0 & 1 & 1 & 1\\
\end{bmatrix}\quad \rightarrow\quad 
\begin{matrix}
1 & 2 &  3 & 4 & 5 & 6\\
2 & 3 &  1 & 5 & 6 & 4\\
3 & 1 & 2 & 6 & 4 & 5\\
\end{matrix}\quad \rightarrow\quad\hspace{-125pt}
\vcenter{\includegraphics[scale=.6]{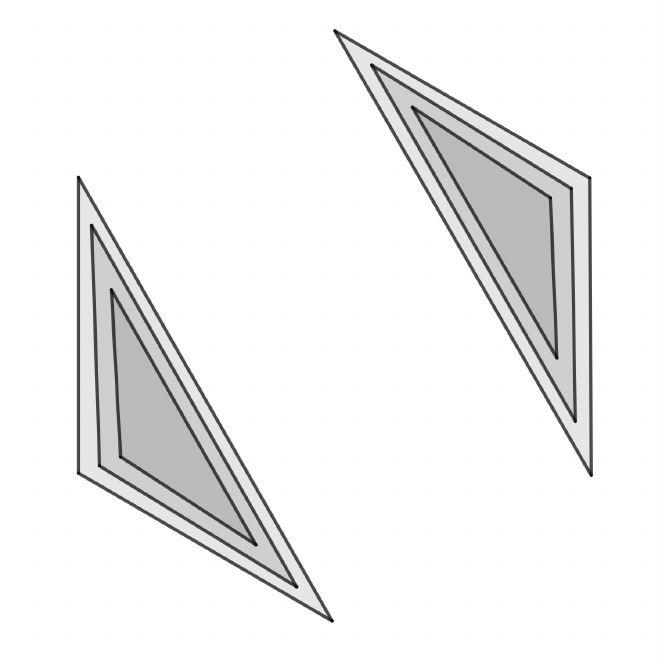}}$$
$$VI:\ \  \begin{bmatrix}
1 & 1 &  0 & 0 & 1 & 0\\
1 & 1 &  0 & 0 & 0 & 1\\
1 & 0 & 1 & 1 & 0 & 0\\
0 & 1 & 1 & 1 & 0 & 0\\
0 & 0 & 1 & 0 & 1 & 1\\
0 & 0 & 0 & 1 & 1 & 1\\
\end{bmatrix}\quad \rightarrow\quad 
\begin{matrix}
1 & 2 &  3 & 4 & 5 & 6\\
2 & 1 &  4 & 3 & 6 & 5\\
3 & 4 & 5 & 6 & 1 & 2\\
\end{matrix}\quad \rightarrow\quad\hspace{-125pt}
\vcenter{\includegraphics[scale=.6]{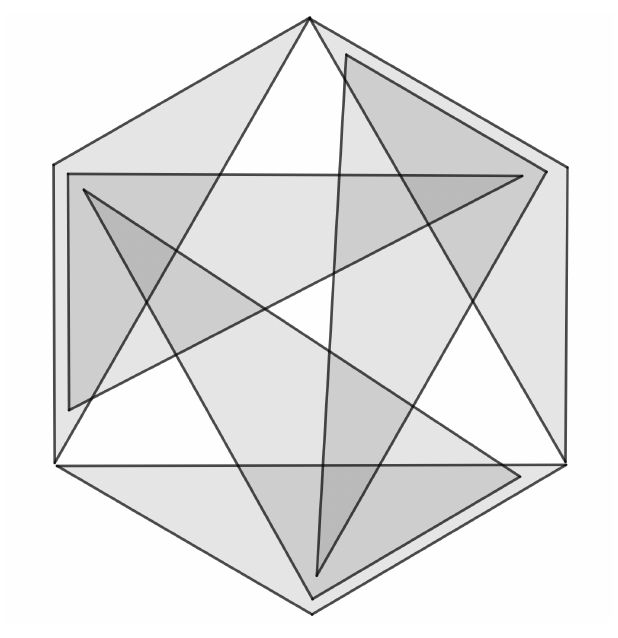}}$$
\caption{Orbits for rank $3:$\ cases $IV-VI$}
\end{figure}

\section{Orbit Data for Rank 3}

We now consider the rank 3 orbits.   There are six orbits, and all but one pair correspond to  self-dual hypergraph classes.  The hypergraphs are given in Figures 5 and 6; vertices are labeled with a 1 at the top and increasing clockwise.  For purposes of verification, the representative $M$ is chosen to exhibit as much symmetry as possible in the corresponding hypergraph; we include the corresponding Latin rectangle for reference. 

\begin{theorem}
The orbits of $M(6,1)$ in rank 3 are given by the equivalence classes in Figures 5 and 6.  Only the pair $\{Ia, Ib\}$ has distinct hypergraph classes under transpose.  These classes may be distinguished by the following features:
\begin{enumerate}
\item [$Ia$]:  one pair of repeated triangles, disjoint from a three-cycle of double edges,
\item [$Ib$]:  three double edges at a vertex, disjoint from a unique triple edge,
\item [$II$]:  two pairs of adjacent double edges,
\item [$III$]: two pairs of repeated triangles, each containing a triple edge,
\item [$IV$]: a six-cycle of double edges,
\item [$V$]: a disjoint pair of triangle triplets, and
\item [$VI$]: a triplet of  disjoint double edges.
\end{enumerate}
\end{theorem}

\begin{proof}
The last two statements are checked directly, and  the second statement also holds since the stabilizer order is preserved under transpose.   That there are no other classes may be seen by enumerating all possibilities for repeated triangles and shared edges while keeping exactly 3 edges at a vertex.  We see this also by exhaustion when performing the extraction process to compute covering data in rank 4.
\end{proof}
Noting Proposition 10.3, we calculate stabilizers as follows:

$Ia$ and $Ib$:  no elements exchange $H$ and $H^*$, so the stabilizers for both cases are isomorphic to $G_H.$  In $Ia,$ a subgroup isomorphic to $S_3$ arises from permutations of the three cycle of double edges. In turn, the exchange of the nested triangles centralizes this subgroup.  The stabilizer has order 12.

$II$: the double edge adjacencies must be preserved; there are eight symmetries of this edge set, but only 4 come from elements of  $G_H$.  The stabilizer has order 8.

$III$: there are 16 symmetries that preserve each  of the doubly-repeated triangles, normalized by the exchange. The stabilizer has 64 elements.

$IV$: the six-cycle of double edges must be preserved, so $G_H$ is isomorphic to $D_{12}.$ The stabilizer has 24 elements.

$V$:  here the blocks of $M$ are preserved by $(S_3)^4$, further normalized normalized by switching blocks. Of course, each $S_3$ either re-orients a triangle or permutes a triple of nested triangles. The order of the stabilizer is 5184, and finally

$VI:$   $G_H$ preserves the set of three double edges, for a maximum of 48 elements. Only 24 of these arise from  elements of $G_H$, so the stabilizer has 48 elements. Again this can be deduced directly from $M$.
\vspace{5pt}

Finally, considering covering data, we need only modify  the extraction technique for rank 4 over rank 3.  The only change occurs in the final step;  an efficient visual identification of the orbit in rank 3 is to sketch the hypergraph, noting first the occurrence of nested triangles, and, if none, shared edges. Here 322 cases require checking.

For instance, extracting $432615$ from $J-I-P_{(12)(3456)}$ in class $A'$, we see that the resultant class in rank 3 is $Ib$:  

$$
\quad \begin{matrix}
3 & 4 & 6 & 2 & \fbox{1} & \fbox{5}\\
\fbox{4} & 6 &  5 & 3 & 2 & 1\\
5 & \fbox{3} & 1 & \fbox{6} & 4 & 2\\
6 & 5 & \fbox{2} & 1 & 3 & 4\\
\end{matrix}\quad \to\ \qquad 
\begin{matrix}
3 & 5 &  6 & 1 & 4 & 2\\
5 & 6 & 1 & 2 & 3 & 4\\
6 & 4 &  5 & 3 & 2 & 1\\
\end{matrix}
\quad \rightarrow\ \ \ 
\vcenter{\includegraphics[scale=.2]{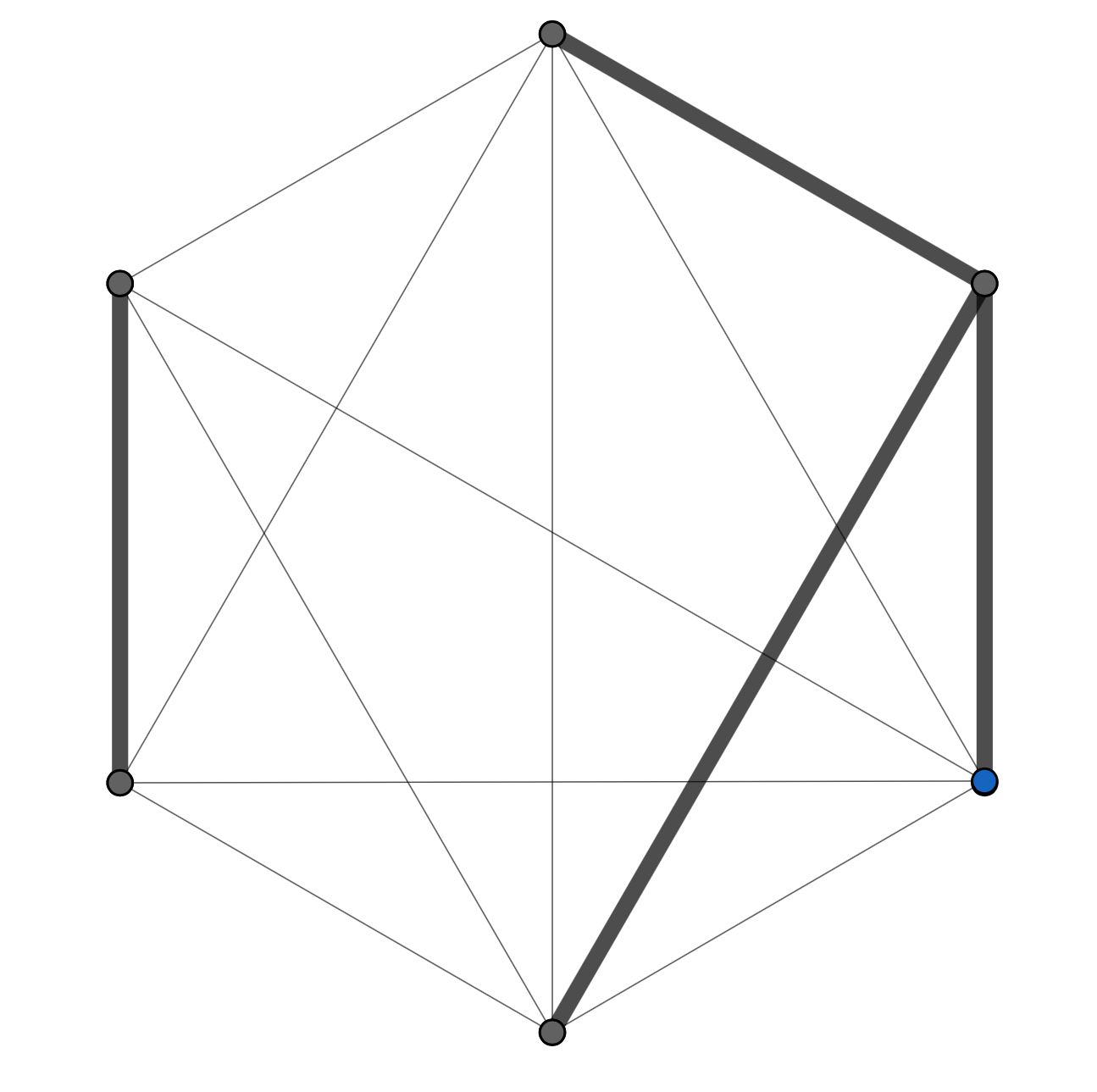}}.
$$

\section{Syzygies and distinct sums}

Finally we briefly consider syzygies in $M(n)$ as a measure the non-uniqueness of $M$ as a sum.  See \cite{StC} for the commutative algebra formulation, although we continue our elementary approach.    A syzygy, or dependence relation with respect to the monoid structure, may be represented by a triple $(M, S_1, S_2),$ where $M$ is in $M(n)$ and $S_i$ is a multiset of permutation matrices that sums to $M$; we allow for  repetitions in $S_i$.  If $\Sigma(S)$ represents the sum of the elements of $S$, then $\Sigma(S_1)-\Sigma(S_2)=0$ represents a non-trivial syzygy if $S_1$ and $S_2$ are distinct.

One special feature of $M(n, 1)$ is that, while there may be several sums that represent a given $M$, the elements in a given sum for $M$ are not repeated. 

\begin{theorem} Let $M$ be in $M(n, 1),$ and let $P(M)$ be the set of subsets of permutation matrices that sum to $M$.  Then the path number
$$v(M) = \rho(M)!\cdot |P(M)|,$$
 and  $|P(M)|$ depends only on the  orbit of $M$.
\end{theorem}
In terms of Latin rectangles as paths, rows may be ordered by increasing first elements.  For the orbits of $M(n, 1)$ for $3\le n \le 5,$  $|P(M)|$ is given in Figure 7.  

Applying the theorem to poset convolution, we obtain
\begin{corollary}  In the notation of Theorem 12.1,  the number of Latin squares of size $n$ is given by 
$$v(J) = (n-k)!\ k!\  \sum\limits_{M}\ o_M\cdot |P(M)|\cdot |P(J-M)|,$$
where $0\le k\le n$ is fixed, $M$ ranges over a set of representatives for the orbits in rank $k,$ and $o_M$ is the number of elements in the orbit for $M$.
\end{corollary}

\begin{figure}[ht]
{\tiny
 \begin{tikzpicture}
  \node (d) at (0, 7) {$2$};
  \node (c) at (0, 6) {$1$};
  \node (b) at (0, 5)  {$1$};
  \node (a) at (0,4) {$1$};
  \draw  (a) edge (b) (b) edge (c) (c) edge (d);
\end{tikzpicture}\hspace{30pt} 
{\tiny
 \begin{tikzpicture}
   \node (e) at (3, 4) {$24$};
  \node (d) at (3,  3) {$4$};
  \node (c1) at (4, 2) {$2$};
  \node (c) at (2, 2) {$1$};
  \node (b) at (3, 1)  {$1$};
  \node (a) at (3,0) {$1$};
  \draw  (a) edge (b) (b) edge (c) (c) edge (d)  (d) edge (e) (b) edge (c1) (c1) edge (d);
\end{tikzpicture}
}\hspace{30pt}
{\tiny
 \begin{tikzpicture}
  \node (e) at (3, 9) {$1344$};
  \node (d2) at (3,  8) {$56$};
  \node (d1) at (1.5,  7) {$6$};
  \node (d) at (4.5,  7) {$4$};
  \node (c1) at (1.5, 6) {$2$};
  \node (c) at (4.5, 6) {$1$};
  \node (b) at (3, 5)  {$1$};
  \node (a) at (3,4) {$1$};
  \draw  (a) edge (b) (b) edge (c) (c) edge (d) (c) edge (d1) (d2) edge (e) (b) edge (c1) (d1) edge (d2) (d) edge (d2) (c1) edge (d1);
\end{tikzpicture}
}
}
\caption{Number of distinct sums $|P(M)|$ for $M$ with $3\le n\le 5$}
\end{figure}
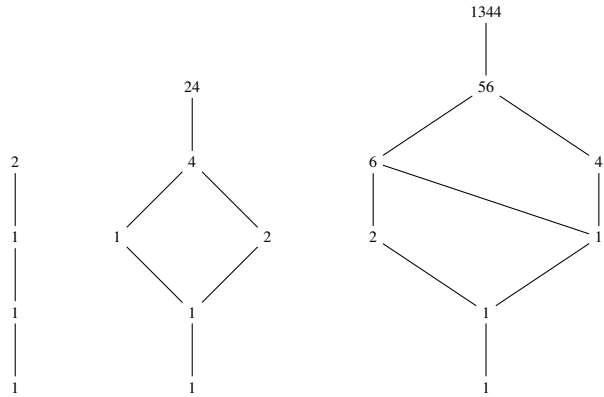

The extraction process of section 5 may be adapted to determine all elements of $P(M)$.  The resulting strings are the only permutations allowed in the construction of a sum for $M$. If these strings are listed lexicographically, we may then implement a sorting method to determine all combinations that sum to $M.$ For instance, with $M=J-P_{(13)(24)}$ in Figure 2, the tree gives the nine strings
$$1234,\ 1243,\ 1324,\ 2134,\ 2143,\ 2341,\ 4123,\ 4231,\ 4321.$$
Using the  tree in Figure 8, we obtain the 4 rectangles that represent $J-P_{(13)(24)}:$
$$
\begin{matrix}
1 & 2 & 3  & 4 \\
2 & 1 & 4 & 3  \\
4 & 3 &  2 & 1 \\
\end{matrix}\qquad 
\begin{matrix}
1 & 2 &  3 & 4\\
2 & 3 & 4 & 1 \\
4 & 1 &  2 & 3 \\
\end{matrix} \qquad
\begin{matrix}
1 & 2 &  4 & 3\\
2 & 1 & 3 & 4 \\
4 & 3 &  2 & 1 \\
\end{matrix}\qquad
\begin{matrix}
1 & 3 &  2 & 4\\
2 & 1 & 4 & 3 \\
4 & 2 &  3 & 1 \\
\end{matrix}.
$$

\begin{figure}
\begin{forest}
for tree={
  l sep=10pt,
  parent anchor=south,
  align=center
}
[
  [1234\\
    [2143\\
      [4321\\     
      ]
       ]
     [2341\\
       [4123\\
       ]
      ]
  ]
  [1243\\
    [2134\\
      [4321\\       
      ]
    ]
  ]  
  [1324\\
    [2143\\
      [4231\\        
      ]
      ]
    ] 
  ]
]
\end{forest}
\caption{Distinct sums for $J-P_{(13)(24)}$ in $M(4, 1)$}
\end{figure}
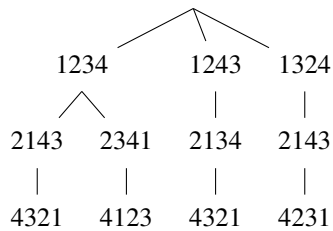

\bibliographystyle{amsplain}

\begin{thebibliography}{99}

\bibitem{BS}
	Billera, L. J., and A. Sarangarajan:
	The combinatorics of permutation polytopes.
	\emph{DIMACS Ser. Discrete Math. Theoret. Comput. Sci.} \textbf{24}, 1--23 (1994)
\bibitem{Bkf}
	Birkhoff, G.:
	Tres observaciones sobre el \'algebra lineal.
	\emph{Univ. Nac. Tucum\'an Rev. Ser. A} \textbf{5}, 147--151 (1946)
\bibitem{BLy}
	Bogart, K. P., and J. P. Longyear:
	Counting 3 by $n$	 Latin rectangles.
	\emph{Proc. Amer. Math. Soc.} \textbf{54},  463--467 (1976)
\bibitem{Bol}
  	Bollob\'as, B.:
	Combinatorics: set sysyems, hypergraphs, families of vectors, and combinatorial probability.
	\emph{Cambridge Univ. Press, Cambridge} (1986)	 	
\bibitem{Bru}
	Brualdi, R. A.:
	Combinatorial matrix classes.
	\emph{Cambridge Univ. Press, Cambridge} (2006)	
\bibitem{DK}
	D\'enes, J., and A. D.  Keedwell:
	Latin squares and their application.
	\emph{Academic Press, New York} (1974)	
\bibitem{Do}
 	Donley, R. W., Jr.:
	Lattice path enumeration for semi-magic squares of size three.
	Preprint, 2021, 15 pages.	 arXiv:\  2107.09463
\bibitem{Ge}
	Gessel, I.:
	Counting Latin rectangles.
	\emph{Bull. Amer. Math. Soc. (N.S.)} \textbf{16}, 79--82 (1987)
\bibitem{LyM}
	Laywine, C. F., and G. L. Mullen:	
	Discrete mathematics using Latin squares.
	\emph{Wiley-Interscience, New York} (1998)
\bibitem{LTT}
	Li, C.-K., Tam, B.-S., and N.-K. Tsing:
	Linear maps preserving permutation and stochastic matrices.
       \emph{Linear Algebra Appl.} \textbf{341}  5--22 (2002)
\bibitem{MKW}
	McKay, B., and I. Wanless:
      	On the number of Latin squares.
	\emph{Ann.  Comb.} \textbf{9}(3), 335--344 (2005)
\bibitem{Phz}
	Paffenholz, A.:
	Faces of Birkhoff polytopes.	
	\emph{Electron. J. Combin.} \textbf{22}(1), \#P1.67 (2015)
\bibitem{Rio}
	Riordan, J.:
	Introduction to combinatorial analysis.
	\emph{Wiley, New York} (1958); reprinted by Dover (2002)
\bibitem{SW}
	Shao, J., and W. Wei:
	A formula for the number of Latin squares.
	\emph{Discrete Math.} \textbf{110}, 293--296 (1992)
\bibitem{Slo} 
	Sloane, N. J. A.:
	The On-Line Encyclopedia of Integer Sequences. Published electronically at http://oeis.org/.	
\bibitem{StC}
	Stanley, R. P.:
	Combinatorics and commutative algebra. Second edition.	Progr. Math. {\bf 41},
	\emph{Birkh\"auser Boston} (2004)
\bibitem{StE}
	Stanley, R. P.:
	Enumerative combinatorics, Volume 1. Second edition.
	Cambridge Stud. Adv. Math. \textbf{49}. 
	\emph{Cambridge University Press, Cambridge}  (2012)
\bibitem{StA}
	Stanley, R. P.:
	Algebraic combinatorics. Second edition.
	Undergrad. Texts Math.
	\emph{Springer Boston} (2018)	
\bibitem{Sto}
	Stones, D. S.:
	The many formulae for the number of Latin rectangles.
	\emph{Electron. J. Combin.} \textbf{17}, \#A1 (2010)	
\bibitem{VN}
	von Neumann, J.:
	A certain zero-sum two-person game equivalent to the optimal assignment problem. In \emph{Contributions to the Theory of Games, Vol. 2}, Ann. Math. Stud. \textbf{28}, 5--12, \emph{Princeton University Press, Princeton} (1953)		
\bibitem{Wa1}
	Wanless, I.:
	Cycle switches in Latin squares.
       \emph{Graphs Combin.} \textbf{20} 545--570 (2004).
	
\end{thebibliography}

\end{document}